\documentclass{amsart}

\usepackage{alphamacros}
\usepackage{analysismacros}
\usepackage{mymathmacros}
\usepackage{refmacros}
\usepackage{formatmacros}

\usepackage{pstricks,pst-plot,pst-math,pst-text, pstricks-add}

\usepackage[colorlinks, backref]{hyperref}

\newcommand{\Cantor}{\ensuremath{\mathfrak{C}}}
\newcommand{\Cantorp}{\ensuremath{\Cantor\pr}}
\newcommand{\Cuts}[1]{\ensuremath{\Cantor_{#1}}}
\newcommand{\Cutinf}{\Cuts{\infty}}
\newcommand{\Sigsum}{\ensuremath{\Sigof{\infseq{x}{i}{1}}}}
\newcommand{\Sigsums}[1]{\ensuremath{\Sigs{#1}}}
\newcommand{\Sigsumsn}{\Sigsums{(n)}}
\newcommand{\seqx}{\single{\xs{i}}}
\newcommand{\seqX}{\single{\Xs{i}}}
\newcommand{\sumx}{\infsum{i}{1}{\xs{i}}}
\newcommand{\setS}{\ensuremath{\mathcal{S}}}

\newcommand{\xsp}[1]{\ensuremath{x^{+}_{#1}}}
\newcommand{\xsm}[1]{\ensuremath{x^{-}_{#1}}}
\newcommand{\ysp}[1]{\ensuremath{y^{+}_{#1}}}
\newcommand{\ysm}[1]{\ensuremath{y^{-}_{#1}}}
\newcommand{\finword}{\ensuremath{\xis{1}\cdots\xis{n}}}

\title{
	Subsum Sets: Intervals, Cantor Sets, and Cantorvals
	}
\author{Zbigniew Nitecki}
\address{Tufts University, Medford, MA 02155}
\email{znitecki@tufts.edu}
\thanks{I would like to thank  Aaron Brown, Keith Burns, Bill Dunham, Richard Kenyon, Micha{\l} Misiurewicz, Don Plante, Jim Propp, Charles Pugh, and Mariusz Urbanski for useful conversations in the course of preparing this paper.}
\date{\today}

\begin{document}

\begin{abstract}
Given a sequence \infseq{x}{i}{1} converging to zero, we consider the set \Sigsum{} 
of numbers which are sums of (infinite, finite, or empty) subsequences of \infseq{x}{i}{1}.  
When the original sequence is not absolutely
summable, \Sigsum{} is an unbounded closed interval which includes zero. 
When it is absolutely summable \Sigsum{}
is one of the following:
a finite union of (nontrivial) compact intervals, a Cantor set, or a ``symmetric Cantorval''.
\end{abstract}

\maketitle

Recently, while trying to think up some challenging problems for my undergraduate 
Real Analysis students, I stumbled onto an elementary and, I think, natural 
question on which I was unaware of any literature.    

One of the most counterintuitive facts in the elementary theory of series is that,
even if a sequence of real numbers \seqx{} converges to zero
(that is, it is a \textbf{null sequence}), the corresponding \emph{series}
\sumx{} might diverge.  The example of this which most of us encounter first is the \emph{harmonic sequence}
	$\recip{k}$, which converges to zero, but whose sum, surprisingly, diverges:
	$\sum_{k=1}^{\infty}\recip{k}=\infty$.
However, if we throw away enough of these terms---for example, if we throw away all reciprocals
of numbers which are not powers of two---we end up with a sequence whose corresponding series
\emph{does} converge.  We will call such a sequence a \emph{summable subsequence} 
of our original sequence, and its sum a \textbf{subsum}, of our (original) sequence.  
Then we might
ask about the set of all possible subsums of a given sequence (assuming always that 
the original sequence goes to zero): is it an interval, a finite union of intervals...
or something more complicated? 

This turns out to be a challenging question: I set out trying to answer it 
and came to a number of interesting conclusions, but was unable to give a satisfactory 
general description of such sets on my own.  However, a comment  
by Micha\l{}  Misiurewicz led me by chance to 
a 1988 paper by J. A. Guthrie and J. E. Nymann \cite{Guthrie-Nymann}, 
which gives a complete topological description of subsum sets as well as a review of some earlier 
work on the problem.%
\footnote{I was gratified to discover that the terminology I had adopted in my musings on the subject 
is almost identical to that used in most earlier writing on the subject.  The one substantial exception is the
word ``Cantorval'', coined by two Brazilians in \cite{Mendes-Oliveira}, which evokes for me
the Samba on Fat Tuesday in Rio...}
After writing up what I had found, I came across the paper of Rafe Jones \cite{RJones}, which covers some of the same 
material, but seems unaware of the definitive result of \cite{Guthrie-Nymann}.  However, it goes beyond
the assumption that the sequence converges to zero.  At the end of the present paper, I will sketch some of the
extensions suggested in Jones' paper, as well as the further extensions in the work of Mor\'an \cite{Moran1, Moran2}
referenced there.

Our story involves an interesting interplay between standard topics on sequences and series, 
some elementary number theory, and the topology of subsets of the line, which provides 
an appealing ``extra topic'' for undergraduate analysis students.

Most of our discussion will focus on \emph{positive} null sequences, which can be studied using geometric
ideas.  However, we shall see toward the end of this paper (\refer{sec}{vary}) 
how the description of \emph{all}  subsum sets
can be reduced to the corresponding description of subsum sets for positive null sequences.

\section{Positive, Non-Summable Sequences}\label{sec:posdiv}
We can think of the harmonic sequence as an infinite collection of 
dominoes of successively shorter lengths:  the $k^{th}$ domino has length \recip{k}.
The fact that the series diverges means that if we put them all end-to-end, 
we will fill out a whole half-line.  

Now suppose we are given a positive real number $r$.
Can we find a collection of dominoes from this set which exactly fill up an interval of length $r$?  

Well, we know the lengths of the dominoes converge monotonically to zero, so except for the first few, 
they are all shorter than any specified fraction of $r$.  This means that 
we can, by starting far enough down the line, 
fit a string of \emph{any specified finite number of successive
dominoes} inside the interval.  If we start with the $n^{th}$ domino and fit in as many
successive dominoes as we can (starting from the $n^{th}$), then the first domino that ``pokes out''
will certainly be shorter than \recip{n}.  In fact, if we have managed to squeeze in $k$ dominoes (starting from the $n^{th}$) but cannot fit the next one in, 
then the one that pokes out has length \recip{n+k}.
This means that the ones we \emph{can} fit fill an interval 
that is \emph{shorter} than $r$---but its length \emph{plus \recip{n+k}} is \emph{more} than $r$.
It follows that after we have squeezed in $k$ successive dominoes starting from the $n^{th}$,
we are left with an unfilled gap which is shorter than \recip{n+k}.
Now, we look for more dominoes, to fill this gap.  We start further down
our list of dominoes, finding a set of $k\pr$ successive ones, starting with the $(n\pr)^{th}$ 
(where $n\pr\gg n+k$), that fill out our gap---except for a new, smaller gap of length less than 
\recip{n\pr+k\pr}.  And we continue.  With a little more care, we can choose our starting
point at each stage so that the size of the gap is cut to less than half its current value with each
new filling.  When we are all done, we have created a subsequence of our dominoes whose 
combined total length is \emph{exactly} $r$.

Let's look back at what we did.  We didn't really use any special properties of the harmonic 
sequence in this construction, other than the fact that the lengths of the dominoes go to zero,
but their sum diverges (to infinity).  So we have a theorem:
\begin{theorem}\label{thm:posdiv}
	If \seqx{} is a positive null sequence 
	for which $\sumx=\infty$,
	then every $r>0$ is the sum of some subsequence of \seqx{}.
\end{theorem}

Actually, there is one minor technical point we need to note here.  When thinking about the 
harmonic sequence, we \emph{did} take advantage of the fact that it is decreasing. 
In general, the sequence we are looking at might be presented in an order which is not
decreasing.  Fortunately ,
for a sequence of \emph{positive} numbers, the sum (of the series) is not changed by rearranging
their order. (This was noted by Dirichlet in 1837 \cite[p. 315]{Dirichlet:Beweis} without explicit proof;
a proof can be found in many basic analysis books, for example \cite[Thm. 3.56, p. 68]{Rudin}.%
\footnote{The basic idea is that the partial sums for any ordering are themselves a strictly increasing
sequence, and any particular partial sum for \emph{one} ordering can be bracketed between two
partial sums of any \emph{other} particular order, so the two limits are the same.}).
Intuitively, the total length of a collection of dominoes set end-to-end is not changed 
if we set them down in a different order.  
This means we can work with them in a \emph{non-increasing} order: 
$\xs{k+1}\leq\xs{k}$ for every $k$.  This will be an implicit assumption in all of our reasoning, at
least while looking at positive sequences:
\begin{quote} \textbf{Standing Assumption:}
	\textit{When dealing with positive sequences, we assume (without loss of generality) that the
	given sequence is non-increasing:}
\end{quote}
\begin{equation*}
	\xs{i+1}\leq\xs{i}\text{ for all }i=1,2,\dots.
\end{equation*}

\section{Positive Summable Sequences}\label{sec:posgeom}
OK, so we have answered our question for a sequence of \emph{positive} numbers going to zero
whose sum \emph{diverges}.  What about if the sum \emph{converges}?

We start with two examples.  

First, consider the sequence of (positive integer) powers of \half
\begin{equation*}
	\xs{i}=\recip{2^{i}},\quad i=1,2,\dots
\end{equation*}
which sums to 
\begin{equation*}
	\sumx=\frac{1/2}{1-1/2}=1.
\end{equation*}
We can again picture our sequence as a collection of dominoes 
(the \ith{} has length $\left(\half\right)^{i}$);  clearly, since \emph{all} of them placed end-to-end 
fill an interval of length $1$, any subcollection will fill a shorter interval;  
that is, any subsum belongs to the interval \clint{0}{1}.
Now, expressing a number in \clint{0}{1} as a sum of (distinct) powers of \half{} 
is the same as giving its binary or \emph{base 2} expansion: to be more precise, a binary sequence
\begin{align*}
	\xi&=\single{\xis{i}}_{i=1}^{\infty}\\
	\intertext{(each \xis{i} is $0$ or $1$) corresponds to the number} 
	\xof{\xi}&=\sum_{i=1}^{\infty}\frac{\xis{i}}{2^{i}}.
\end{align*}
Every number between $0$ and $1$ has a binary expansion, so the subsum set in this case%
\footnote{We shall see later that this needs some clarification: see \refer{dfn}{subsumset}.}
 is an interval with endpoints $0$ and $1$.

Now, consider the sequence of powers of \recip{3}
\begin{align*}
	\xs{i}&=\recip{3^{i}}, \quad i=1,2,\dots\\
	\intertext{which sums to}
	\sumx&=\frac{1/3}{1-1/3}=\half.
\end{align*}
As before, any subsum belongs to the interval \clint{0}{\half}.  But on closer inspection, it becomes
clear that not every point in this interval occurs as a subsum.  For example, any subsum which
does \emph{not} involve the first term, \recip{3}, is at most equal to
\begin{align*}
	\Xs{1}\eqdef&\sum_{i>1}\recip{3^{i}}=\frac{1/9}{1-1/3}=\recip{6}\\
	\intertext{and hence belongs to the interval}
	\Js{0}&\eqdef\clint{0}{\recip{6}}\\
	\intertext{whereas any subsum which \emph{does} involve the first term belongs to}
	\Js{1}&\eqdef\clint{\recip{3}}{\half}.\\
	\intertext{Note that \Js{1} is the translate of \Js{0} by $\xs{1}=\recip{3}$, and the set of 
	subsums is actually contained in the union of two disjoint closed intervals}
	\Cuts{1}&\eqdef\Js{0}\cup\Js{1}.
\end{align*}
That is, distinguishing subsums according to whether they do or don't involve the first term
of the sequence breaks the set of all subsums into two pieces, the second a translate of the first. 
When we take account of all the possibilities for which of the first \emph{two} terms of the sequence
occur in a given subsum, we find that the set of subsums is contained in the union of four disjoint
closed intervals--two subintervals of \Js{0} and two subintervals of \Js{1}.
Of course, we can continue this process.  
A subsequence 
of \seqx{} can
be specified using the sequence $\xi=\xis{1}\xis{2}\cdots$ of zeroes and ones defined by
\begin{equation}\label{eqn:zeroone}
	\xis{k}=\begin{cases}
			   1   & \text{ if \xs{k} is included in the subsequence }, \\
			    0  & \text{ if it is not}.
		\end{cases}
\end{equation}
The sum corresponding to this subsequence is then
\begin{equation}\label{eqn:expansion}
	\xs{\xi}\eqdef\sum_{k=1}^{\infty}\xis{k}\cdot\xs{k}.
\end{equation}
For our particular example, this reads
\begin{equation*}
	\xs{\xi}=\sum_{k=1}^{\infty}\frac{\xis{k}}{3^{k}}
\end{equation*}
which is a base three expansion for \xs{\xi}.  

The intervals \Js{00},\Js{01},\Js{10} and \Js{11} result from sorting the
subsum set according to which of the first two terms of the sequence $\seqx=\infseq{x}{i}{1}$ 
are included in a given subsum--that is, according to the initial ``word'' of length $2$ in 
the defining sequence \xie.  In general, 
we can parse any subsum into an initial finite sum, \xs{\finword} determined by
the initial ``word'' $\finword$ of length $n$, and the rest of the sum, which is 
a subsum of the sequence \infseq{x}{i}{n+1} obtained by omitting the first $n$ terms of \infseq{x}{i}{1}.
Let us informally%
\footnote{A formal definition will be given shortly in \refer{dfn}{subsumset}.}
denote the subsum set of a sequence \infseq{x}{i}{1} by \Sigsum{}, and write
\begin{equation*}
	\Sigsumsn{}\eqdef\Sigof{\infseq{x}{i}{n+1}}
\end{equation*}
for the set of subsums which do not involve the first $n$ terms $\xs{1},\dots,\xs{n}$.
Then the collection of all subsums whose defining sequence \xie{} has initial word $\finword$
can be written
\begin{equation*}
	\xs{\finword}+\Sigsumsn{};
\end{equation*} 
letting the initial word of length $n$ range over all the possible \tuple{n}s 
of zeroes and ones, we fill up our subsum set:
\begin{equation}\label{eqn:partition}
	\Sigsum=\bigcup_{\finword\in\single{0,1}^{n}}
		\left(\xs{\finword}+\Sigsumsn{}\right).
\end{equation}

As before, \Sigsumsn{} is contained in the closed interval%
\footnote{$0^{n}$ denotes the word of length $n$ consisting of all zeroes.}
\begin{align*}
	\Js{0^{n}}&=\clint{0}{\Xs{n}}\\
	\intertext{where \Xs{n} is the highest sum in \Sigsumsn{}}
	\Xs{n}&=\sum_{k>n}\xs{k}\\
	\intertext{and it follows that (for each fixed $n$) our whole subsum set is contained 
		in the union of $2^{n}$ closed intervals}
	\Cuts{n}&=\bigcup_{\finword\in\single{0,1}^{n}}\Js{\finword}\\
	\intertext{where}
	\Js{\finword}&\eqdef\xs{\finword}+\Js{0^{n}}\\
		&=\clint{\xs{\finword}+0}{\xs{\finword}+\Xs{n}}.
\end{align*}

In our example, 
\begin{equation*}
	\Xs{n}=\sum_{k=n+1}^{\infty}\recip{3^{k}}=\frac{1/3^{n+1}}{1-\recip{3}}
		=\recip{2\cdot3^{n}}
\end{equation*}
so 
\begin{equation*}
	\Js{0^{n}}=\clint{0}{\recip{2\cdot3^{n}}}.
\end{equation*}

Having fixed an initial word of length $n$, we have two possibilities for the next, \st{(n+1)} entry in \xie: 
either $\xis{n+1}=0$ or $\xis{n+1}=1$.
This means that each interval \Js{\finword} of \Cuts{n} contains two subintervals associated to
initial words of length $n+1$ in \xie:
\begin{align*}
	\Js{\finword0}&=\xs{\finword}+\clint{0}{\Xs{n+1}}\\
	\intertext{and}
	\Js{\finword1}&=\xs{\finword}+\recip{3^{n+1}}+\clint{0}{\Xs{n+1}}\\
\end{align*} 
where
\begin{equation*}
	\Xs{n+1}
		=\recip{2\cdot3^{n+1}}.
\end{equation*}
The important thing to notice is that these two subintervals have the same length, \Xs{n+1}, 
and the second is a translate of the first by an amount greater than \Xs{n+1}. This means they are disjoint.
Looking a bit more closely, we note that the first subinterval \emph{starts} 
at the \emph{left} endpoint of \Js{\finword} while the second \emph{ends} at its \emph{right} endpoint.  
Thus, passing from the union \Cuts{n} of intervals determined by words of length $n$ to the union
\Cuts{n+1} of those determined by words of length $n+1$, each component interval of \Cuts{n} 
acquires a gap in its middle, separating two subintervals which are components of \Cuts{n+1}.
In fact, since $\Xs{n+1}=\recip{3}\Xs{n}$, this gap is precisely the ``middle third'' of each component.
Hence we are carrying out the construction of the middle-third Cantor set, except that we start from the interval
\clint{0}{\half} instead of \clint{0}{1}.  In this way, when we pass to the intersection 
\begin{equation*}
	\Cutinf{}=\bigcap_{n=1}^{\infty}\Cuts{n}
\end{equation*}
we obtain a version of the standard Cantor set, 
but scaled down by a factor of a half.

The argument above shows that the subsum set of the sequence of powers of \recip{3} is a Cantor set.
However, the construction of the sets \Cuts{n} and \Cutinf{} applies to \emph{any} positive summable null
sequence, with the proviso that in general, the intervals \Js{\finword} need not be disjoint---so our 
final set \Cutinf{} need not be a Cantor set.  In fact, for the powers of \half, we have
\Xs{n}=\recip{2^{n}}, and the intervals \Js{\finword} abut, so $\Cuts{n}=\clint{0}{1}$ for all $n$ 
(and hence for ``$n=\infty$'').  As we shall see, even more complicated behavior is possible 
which mixes overlap and disjointness.

In general, though, the procedure we have outlined produces the compact set \Cutinf{}, 
which is guaranteed to contain our subsum set.  
But certainly at each finite stage, the set \Cuts{n}
contains more than \Sigsum.  So, what about the intersection?---does \Sigsum{} \emph{equal} \Cutinf{}, or is it a proper subset? 

The answer to this hinges on what we mean by a ``subsequence''. 
Usually a ``subsequence'' of an infinite
sequence is understood to itself be infinite.  If we use this notion in our definition of subsums, 
we exclude any number given as a \emph{finite} sum of powers of \recip{3}---% 
that is, we exclude the left endpoint of each of our intervals \Js{\xis{1}\cdots\xis{k}}.  
The resulting set is a bit awkward to describe.  So we follow a convention 
going back to S. Kakeya (whose 1914 paper \cite{Kakeya} is the first one I am aware of on this topic) and
include \emph{finite} subsequences, as well as the \emph{empty} sequence 
(whose sum we take to be zero), in our formal definition of the subsum set.
\begin{definition}\label{dfn:subsumset}
	The \deffont{subsum set} of a null sequence 
	\begin{equation*}
		\xs{1},\xs{2},\dots\to0
	\end{equation*}
	is the collection 
	\begin{equation*}
		\Sigsum{} 
	\end{equation*}
	of all numbers of the form
	\begin{equation*}
		\xs{\xi}\eqdef\sum_{k=1}^{\infty}\xis{k}\cdot\xs{k},
	\end{equation*}
	 where 
	\begin{equation*}
		\xi=\xis{1}\xis{2}\cdots
	\end{equation*}
	is any sequence of zeroes and ones for which the subsequence \finseq{\xis{i}\cdot\xs{i}}{i}{1}{\infty}
	is summable.%
\footnote{In the context of this section, where we have assumed the original sequence is positive 
and summable, every subsequence is summable.}
\end{definition}
This definition simply codifies the idea that we take sums of infinite, finite, and empty subsequences of
\infseq{x}{i}{1}.  Note that a \emph{finite} subsum corresponds to a sequence \xie{} which is 
eventually all zeroes; we shall often omit the ``tail of zeroes'' when specifying the sequence $\xi$ 
in such a situation.

Now, we have constructed a nested sequence of compact sets \Cuts{n}, each containing our subsum set; 
it follows that \Sigsum{} is contained in the compact set \Cutinf{}.  Furthermore, \Cuts{n} consists of 
intervals of length \Xs{n}, each having nonempty intersection with \Sigsum{} (for example its endpoints),
which means that all points of \Cuts{n} are within distance \Xs{n} of the set \Sigsum.  
Since we have assumed our sequence is summable, its ``tails'' \Xs{n} must converge to zero, which implies
that \Cutinf{} is the \emph{closure} of \Sigsum.  

The construction of \Cutinf{} automatically implies several properties:
\begin{itemize}
 	\item Since $\Xs{n}>0$, every component of  \Cuts{n} is an interval, and hence it has no isolated 
	points---it is \emph{perfect}.  This property persists under nested intersection, so \Cutinf{}
	is a perfect set.
	\item \Cuts{n} is a union of closed intervals \Js{\finword} of length \Xs{n}; in particular,
	each point of \Cuts{n} is within distance \Xs{n} of at least one \emph{right} endpoint and at least
	one \emph{left} endpoint of some \Js{\finword}.  Since $\Xs{n}\to0$,
	this means the right \resp{left} endpoints of the various intervals \Js{\finword} are
	dense in \Cutinf{}.
\end{itemize}

In the case of powers of \recip{3}, we have an explicit homeomorphism between the subsum set 
\Sigof{\single{\recip{3^{k}}}_{k=1}^{\infty}} and the middle-third Cantor set, telling us that this subsum
set is compact, and hence \emph{equals} \Cutinf{}.  To go beyond this example, 
we need to show that \Sigsum{} is 
closed in general.  This was done in \cite{Kakeya} by a direct argument, 
but we can finesse the general case using the example and a sneaky
trick.  

For our example (powers of \recip{3}), the sequence \xie{} for a particular subsum is a base $3$
expansion of that subsum, so points of the Cantor set are in one-to-one correspondence with sequences
\xie{} of zeroes and ones. Furthermore, this mapping is a homeomorphism (points with expansions that
agree for a long time are close to each other, and vice-versa).  But half of this also applies to a general
subsum set:  for \emph{any} sequence \seqx{}, two subsums whose defining sequences \xie{} agree for
at least $n$ places belong to the same interval \Js{\finword}, which is an interval of length \Xs{n}.
And that length, which is by definition a ``tail'' of a convergent series, goes to zero.  Thus, the mapping
taking a point of the (middle-third) Cantor set to its defining sequence \xie{} and then to the point \xs{\xi}
in our subsum set corresponding to the same sequence is continuous.  Since it is also onto, we have 
exhibited a general subsum set \Sigsum{} as a continuous image of a compact set---hence it is also compact.  From this we can conclude that
\begin{equation*}
	\Cutinf{}=\Sigsum.
\end{equation*}
Hence \Sigsum{} has the properties noted above for \Cutinf{}: it is a perfect set, and (since the left
\resp{right} endpoint of any \Js{\finword} is the sum of a finite \resp{infinite} subsequence),
both kinds of sums are dense in \Sigsum.

\Sigsum{} also has some symmetry properties.  We have already seen (fixing $n$) that  \Cuts{n} is a 
union of sets \Js{\finword} which are just translates of each other;  this means that 
for each fixed $n$ the sets $\Sigsum\cap\Js{\finword}$ are homeomorphic.  Another symmetry is the
reflection about the midpoint, given by
\begin{equation}\label{eqn:reflect}
	x\mapsto\Xs{0}-x.
\end{equation}  
To see this particular symmetry, note that when $x$ is a subsum of our sequence defined by the sequence
\xie{} of $0$'s and $1$'s, then $\Xs{0}-x$ is defined by the sequence $\tilde{\xi}_{i}$, where
$\tilde{\xi}_{i}=1-\xis{i}$---that is, $\Xs{0}-x$ is the sum of all the terms \emph{not} included in the sum
defining $x$.

We summarize%
\footnote{No pun intended.} 
these general observations in the following theorem:
\begin{theorem}\label{thm:subsumset}
 	For every summable, positive null sequence $\xs{1},\xs{2},\dots\to0$ with sum
	\begin{equation*}
		\sum_{k=0}^{\infty}\xs{i}=\Xs{0},
	\end{equation*}
	the subsum set \Sigsum{} is a perfect set with convex hull \clint{0}{\Xs{0}} which is 
	symmetric under the reflection $x\mapsto\Xs{0}-x.$
	
	Furthermore, the collection of all sums of finite subsequences (as well as the collection
	of all sums of infinite subsequences) is dense in \Sigsum.
\end{theorem}

The fact that \Sigsum{} is perfect was proved by Shoichi Kakeya in 1914 \cite{Kakeya} and independently by Hans Hornich in 1941 \cite{Hornich}%
\footnote{A 1948 paper by P. Kesava Menon \cite{Menon} addresses similar issues, but I find it confusing to determine just what is being proved.}  The reflection symmetry of subsum sets 
was noted by Hornich, as well as by Joseph Nymann and Ricardo Saenz in \cite{Nymann-Saenz}. 

\section{Terms vs. Tails: Subsum sets of geometric and $p$-series}\label{sec:termstails}
In the examples studied so far, we have observed two extremes of behavior.  
For the powers of \recip{3},  the intervals \Js{\finword} for any fixed $n$ are disjoint, and in the limit we
obtain a Cantor set as \Cutinf.  But for powers of \recip{2}, these intervals touch, as a result of which all
the sets \Cuts{n} are the same, and \Cutinf{} is an interval.  To understand the basis of these phenomena
in general, we examine the recursive step in the construction of \Cutinf.

When we go from \Cuts{n-1} to \Cuts{n}, each interval \Js{\xi} (for a fixed
$(n-1)$-word \xie) is replaced by  the union of two subintervals, 
corresponding to the $n$-words $\xi^{-}=\xi0$
and $\xi^{+}=\xi1$ obtained by appending $0$ \resp{$1$} to \xie.  
Both of these subintervals have length equal to 
the \nth{} \emph{tail} \Xs{n}, and the second is the translate of the first by the \nth{} \emph{term} 
\xs{n}.  Furthermore, the right \resp{left} endpoint of \Js{\xi}
is the same as the right \resp{left} endpoint of \Js{\xi^{-}} \resp{\Js{\xi^{+}}}.
Thus we can distinguish two cases:
\begin{description}
	\item[Term exceeds Tail] 
	If $\xs{n}>\Xs{n}$, the two intervals are disjoint, so \Js{\xi} in \Cuts{n} is replaced
	by a \emph{disjoint} union of subintervals in \Cuts{n+1}; that is, \Js{\xi} breaks into the 
	\emph{disjoint} union of \Js{\xi^{-}} and \Js{\xi^{+}}, leaving a ``gap''  of size $\xs{n}-\Xs{n}$
	in the middle.
	\item[Tail bounds Term]
	 If $\xs{n}\leq\Xs{n}$, the two intervals share at least one point, so their union equals
	\Js{\xi}.
\end{description}
Note that with $n$ fixed, the intervals \Js{\xi} obey the same rule for every $n$-word. 
Also remember that they all have length \Xs{n}, which goes to zero (since \Xs{n} is the tail of a 
convergent series). Note also that in the first case, \Cuts{n+1} is obtained from \Cuts{n} by deleting an interval
of length $\xs{n+1}-\Xs{n+1}$ from each of its $2^{n}$ components. Since these all have length \Xs{n}, the
total length of \Cuts{n+1} is $2^{n}\{\Xs{n}-(\xs{n+1}-\Xs{n+1})\}=2^{n+1}\Xs{n+1} $.
This shows
\begin{theorem}\label{thm:event}
	Suppose \seqx{} is a summable sequence of positive real numbers.  Then
	\begin{enumerate}
		\item\label{thm:event1}
		 If $\xs{n}>\Xs{n}$ (\ie{} the term exceeds the tail) for every $n$, 
		 then for each $n$, \Cuts{n} is the disjoint union of the $2^{n}$ closed intervals \Js{\xi}
		 as \xie{} ranges over the words of length $n$.  It follows that
		 $\Cuts{\infty}=\Sigsum$ is a Cantor set whose Lebesgue measure is
		 $
		 	\lim_{n\to\infty} 2^{k}\Xs{k}.
		 $
		\item\label{thm:event2}
		 If $\xs{n}\leq\Xs{n}$ (\ie{} the tail bounds the term) for every $n$, 
		 then for each $n$, $\Cuts{n}=\Cuts{0}=\clint{0}{\Xs{0}}$, so
		 $\Cuts{\infty}=\Sigsum$ is the interval
		\clint{0}{\Xs{0}}.
	\end{enumerate}
\end{theorem}
These properties were established by Hornich \cite{Hornich}.  Kakeya \cite{Kakeya} noted the second property
(in fact that the tail always bounds the term if and only if the subsum set is an interval---cf. 
our \refer{lem}{gaps} and \refer{prop}{gaps}).

For a geometric sequence with first term $a$ and ratio%
\footnote{(that is, a geometric sequence whose terms are positive and tend to zero)} $\rho\in\opint{0}{1}$, we know that
\begin{align*}
	\xs{n}&=a\rho^{n-1}\\
	\intertext{and}
	\Xs{n}&=\frac{a\rho^{n}}{1-\rho}\\
	\intertext{so}
	\frac{\Xs{n}}{\xs{n}}&=\frac{\rho}{1-\rho}
\end{align*}
which is at least $1$ for $\rho\geq\half$ and strictly less than $1$ for $0<\rho<\half$.
This immediately gives us a description of \Sigsum{} for any positive geometric sequence.%
\footnote{Jones \cite[Prop. 3.3]{RJones} gives a kind of extension of the first case of this corollary, in the spirit of the
ratio test for convergence.}
\begin{corollary}
	If \single{\xs{i}=a\rho^{i-1}} is a geometric sequence with initial term $a>0$ and ratio
	$\rho\in\opint{0}{1}$, then \Sigsum{} is
	\begin{enumerate}
		\item a Cantor set of measure zero for $0<\rho<\half$
		\item the interval \clint{0}{\Xs{0}=\frac{a}{1-\rho}} for $\half\leq\rho<1$.
	\end{enumerate}
\end{corollary}

\refer{thm}{event} tells us what happens when only one of the two possible relations 
between the terms and the tails occurs.  What about if both occur, but one of them occurs \emph{eventually}?%
\footnote{A property is said to hold \emph{eventually} for a sequence if there is some place $K$ in the 
sequence so that the property holds for \emph{all} later terms---or equivalently, if the property fails to
hold for at most a finite number of terms. } 

As an example, consider the sequence starting with $2$ and then followed by the powers of \half.
We already know that the sequence starting from the second term (\ie{}, just the powers of \half{}) 
has subsum set \clint{0}{1}, and it follows from \refer{eqn}{partition} that the full subsum set is
\begin{equation*}
	\Cutinf=\clint{0}{1}\cup\left(2+\clint{0}{1}\right) =\clint{0}{1}\cup\clint{2}{3}.
\end{equation*}
These two intervals are disjoint because the first term, $\xs{1}=2$, is greater than the first tail, $\Xs{1}=1$.

In general, if the tail bounds the term after the $N^{th}$  place
\begin{equation*}
	\xs{k}\leq\Xs{k}\text{ for }k>N
\end{equation*}
then \refer{thm}{event2} applied to the sequence starting after position $N$ tells us that
\begin{equation*}
	\Sigsums{(N)}=\clint{0}{\Xs{N}}
\end{equation*}
and then \refer{eqn}{partition} tells us that \Sigsum{} is the union of $2^{N}$ closed intervals,
which means, allowing for some overlaps between them, that it is the \emph{disjoint} union of 
\emph{at most} $2^{N}$ intervals. 
Furthermore, if the term exceeds the tail for \emph{all} of the first $K$ places
\begin{equation*}
	\xs{k}>\Xs{k}\text{ for }k=1,2,...,K
\end{equation*}
then the intervals \Js{\finword} are all disjoint, so \Cuts{K} consists of $2^{K}$ disjoint intervals. 
So in this case $\Cutinf=\Cuts{N}$ has \emph{at least} $2^{K}$ components.
Summarizing, we have

\begin{prop}\label{prop:eventtailbds}
	Suppose \seqx{} is a positive, summable null sequence.
	\begin{enumerate}
		\item\label{prop:eventtailbds1} 
		If the tail bounds the term eventually, then \Cutinf{} is a finite union of closed intervals.
		
		\item\label{prop:eventtailbds2} 
		In particular, if the tail bounds the term for all $k>N$ then $\Cutinf=\Cuts{N}$  consists of at most
		$2^{N}$ disjoint closed intervals.  
		\item\label{prop:eventtailbds3} 
		If in addition the term exceeds the tail for $k=1,\dots,K$,
		then \Cutinf{} consists of at least $2^{K}$ disjoint closed intervals.
	\end{enumerate}
\end{prop}

As an example, consider  the \emph{$p$-sequence}
\begin{equation*}
	\xs{k}=\recip{k^{p}}
\end{equation*}
where $p>1$ is a fixed exponent.  The precise value of the \nth{} tail \Xs{n} is hard to determine,
but we can take advantage of the standard proof of summability (that is, the integral test) to estimate
it and so try to check which terms exceed the associated  tails and which tails bound the terms.  

	\begin{figure}[htbp]
	\begin{center}
		\begin{pspicture}(-0.5,-0.5)(4.0,5.0)
			\psline{<->}(0,5)(0,0)(4.0,0)
			\psplot[linewidth=2pt]{1}{2.5}{3 x dup mul div}
			\psplot[linestyle=dashed,linewidth=2pt]{2.5}{4.0}{3 x dup mul div}
			\psplot[linestyle=dashed,linewidth=2pt]{0.8}{1.0}{3 x dup mul div}
			
			\pspolygon[linestyle=none,fillstyle=solid, fillcolor=lightgray]%
				(1,0)(1,1.33)(1.5,1.33)(1.5,0.75)(2,0.75)(2,0.48)(2.5,0.48)(2.5,0)
			\psline(1,0)(1,3)(1.5,3)(1.5,1.33)
			\psline(1.5,0)(1.5,1.33)(2,1.33)(2,0.75)
			\psline(2,0)(2,0.75)(2.5,0.75)(2.5,0.48)
			\psline(2.5,0)(2.5,0.48)
			
			\uput[r](1.5,3){\recip{k^{p}}}
			\uput[r](2.0,1.33){\recip{(k+1)^{p}}}
			\uput[dl](1,0){\recip{k^{p}}}
			\uput[d](1.5,0){\recip{(k+1)^{p}}}
			\rput(3.0,4.0){$y=\recip{x^{p}}$}
		\end{pspicture}	
	\caption{Integral Test for $p$-series}
	\label{fig:integ}
	\end{center}
	\end{figure}
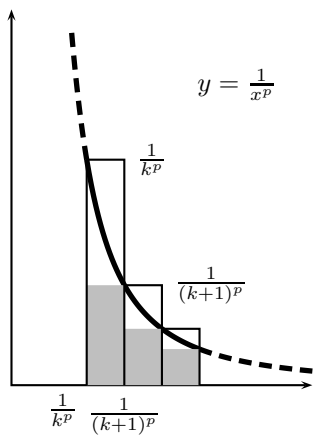

From \refer{fig}{integ} we obtain the estimates
	\begin{equation*}
		\int_{n+1}^{\infty}\frac{dx}{x^{p}}
		<\sum_{k=n+1}^{\infty}\recip{k^{p}}
		<\int_{n}^{\infty}\frac{dx}{x^{p}}.
	\end{equation*}
Carrying out the integration on 
either side, we have 
	\begin{equation}\label{eqn:estimate}
		\frac{(n+1)^{1-p}}{p-1}<\Xs{n}<\frac{n^{1-p}}{p-1}.
	\end{equation}
Thus we can guarantee that the \nth{} tail exceeds the \nth{} term
	\begin{align*}
		\xs{n}&<\Xs{n}\\
		\intertext{whenever}
		\recip{n^{p}}&\leq\frac{(n+1)^{1-p}}{p-1}.
		\intertext{a condition which can be rewritten}
		p-1&\leq(n+1)\left(\frac{n}{n+1}\right)^{p}.
	\end{align*}
Fixing $p>1$, the fraction on the right converges to $1$, while the first factor goes
to infinity, so (for a given exponent $p$) the \emph{the \nth{} tail bounds the \nth{} term eventually}.

We leave it to the reader to check that the function 
\begin{equation*}
	\fsof{p}{x}=(x+1)\left(\frac{x}{x+1}\right)^{p}
\end{equation*}
is strictly increasing.

However, the condition
	\begin{align*}
		\xs{n}&>\Xs{n}\\
		\intertext{is guaranteed to hold whenever}
		\frac{n^{1-p}}{p-1}&\leq\recip{n^{p}}\\
		\intertext{or}
		p-1&\geq n:\\
	\end{align*}
\emph{the \nth{} term exceeds the \nth{} tail at least for}%
\footnote{$\lfloor p\rfloor$ denotes the highest integer $\leq p$} 
$n\leq K\eqdef\lfloor p\rfloor-1$.
We then have
\begin{corollary}
	The subsum set of a summable $p$-sequence is a finite union of disjoint closed intervals.
	The number of these intervals is between $2^{K}$ and $2^{N}$, where
	\begin{itemize}
		\item $K$ is the highest integer less than or equal to $p-1$, and
		\item $N$ is the least integer such that
		\begin{equation*}
			p-1\leq(N+1)\left(\frac{N}{N+1}\right)^{p}.
		\end{equation*}
	\end{itemize}
\end{corollary}

\refer{prop}{eventtailbds} takes care of sequences for which the tail eventually bounds the term.
The situation is more complicated when the term eventually exceeds the tail, but not immediately.

If at some stage the term exceeds the tail, it is still true that each of the intervals \Js{\finword} 
will split into two subintervals separated by a ``gap''.
However, if the tail bounded the term at some previous stage, we can no longer assume that the
intervals which are splitting are disjoint: in principle the ``gap'' created when one of them
splits can be covered over by part of another one, so that (at least as far as this part of the set is concerned)
no new gap is created in \Cuts{n+1}. 

An example of this phenomenon is the sequence
$
	\frac{2}{5},\frac{9}{25},\frac{12}{125},\frac{54}{625},\dots
	$
defined by%
\footnote{We shall see how this mysterious sequence was created in \refer{sec}{bigeom}.}
%\begin{align*}
\begin{equation}\label{eqn:alternate}
	\begin{split}
		\xs{2k}&=\frac{9\cdot6^{k-1}}{5^{2k}}\\
		\xs{2k+1}&=\frac{2\cdot6^{k}}{5^{2k+1}}.
	\end{split}
\end{equation}
This sequence is summable, with
\begin{equation*}
	\Xs{0}=\sum_{k=0}^{\infty}\frac{2\cdot6^{k}}{5^{2k+1}}+\sum_{k=1}^{\infty}\frac{9\cdot6^{k-1}}{5^{2k}}
		=1
\end{equation*}
and the first four tails are
\begin{equation*}
	\Xs{1}=\frac{3}{5},\quad
	\Xs{2}=\frac{6}{25},\quad
	\Xs{3}=\frac{18}{125},\quad
	\Xs{4}=\frac{36}{625}.
\end{equation*}
In general, the \emph{even}-numbered \emph{tails} exceed the corresponding \emph{terms}, 
but the \emph{odd}-numbered \emph{terms} exceed the corresponding \emph{tails}.
In particular, in the passage from \Cuts{2n} to \Cuts{2n+1}, each interval breaks into two overlapping intervals, so
$\Cuts{2n+1}=\Cuts{2n}$.  However, the passage from an odd-numbered set to an even-numbered set is more
complicated: it is still true that each of the intervals \Js{\xi} is replaced by two disjoint subintervals, but the
``gap'' this produces is sometimes covered by one of the  subintervals coming from a different \Js{\xi\pr}.
For example, \Cuts{3} has three components:
\begin{equation*}
	\Js{000}\cup\Js{001}=\clint{0}{\frac{30}{125}},
	\Js{010}\cup\Js{011}\cup\Js{100}\cup\Js{101}=\clint{\frac{45}{125}}{\frac{80}{125}},
	\Js{110}\cup\Js{111}=\clint{\frac{95}{125}}{1}.
\end{equation*}
A straightforward but tedious calculation shows that each of the two end components breaks into 
three components--for example the left component becomes
\begin{equation*}
	\left\{\Js{0000}=\clint{0}{\frac{36}{625}}\right\}\cup
	\left\{\Js{0001}\cup\Js{0010}=\clint{\frac{54}{625}}{\frac{96}{625}}\right\}\cup
	\left\{\Js{0000}=\clint{\frac{114}{625}}{\frac{150}{625}}\right\}
\end{equation*}
--but the middle component remains unchanged. 
Nonetheless, the components of \Cuts{n} do appear to keep breaking up into subintervals,
suggesting that at the end there will be infinitely many components to \Sigsum.
In fact, this turns out to be true.  
To see why, we need to study what happens at the far left of \Cuts{n} when the term exceeds the tail.

Every interval \Js{\xi} 
($\xi$ a word of length $n$) is a translate to the right of the leftmost interval \Js{0^{n}} by \xs{\xi}. 
Since we have assumed the sequence is non-increasing, the shortest of these translations
is given by \xs{n}.  Thus for any $n$ the interval \ropint{0}{\xs{n}} 
intersects \Js{0^{n}}
but is disjoint from all the other intervals \Js{\xi} (\xie{} a word of length $n$) making up \Cuts{n}.
Suppose now that the \nth{} term exceeds the \nth{} tail ($\xs{n}>\Xs{n}$),
so that  \clint{0}{\xs{n-1}} breaks into two subintervals, \clint{0}{\Xs{n}} and \clint{\xs{n}}{\xs{n}+\Xs{n}}.
The only word \xie{} of length $n-1$ for which \Js{\xi} intersects \clint{0}{\xs{n-1}} is $\xi=0^{n}$.
Thus, the ``gap'' \opint{\Xs{n}}{\xs{n}} introduced into \Js{0^{n}} when it breaks up 
into \Js{0^{n}0=0^{n+1}} and \Js{0^{n}1}
becomes a gap in \Cuts{n}.  

A similar argument applies when \xs{\xi} is the left endpoint of some component of \Cuts{n-1};
the easiest way to see this is to translate the whole picture using  \xs{\xi} and note that 
to the left of \xs{\xi} there is a larger ``gap'' coming from some earlier separation.
Finally, we can use the symmetry of \Cuts{\infty} under $x\mapsto\Xs{0}-x$ established
in \refer{thm}{subsumset} to draw the same conclusion for the \emph{right} endpoints of components
of \Cuts{n}.  This gives us

\begin{lemma}\label{lem:gaps}
	Suppose \seqx{} is a positive, non-increasing summable sequence.%
	\footnote{Every positive sequence can be rewritten in non-increasing order 
	without changing the subsum set. However, the sequence of tails---and hence, presumably,
	the times when the term exceeds the tail---is certainly affected by such a reordering.
	It is critical for our argument that the sequence be given in non-increasing order before this
	condition is checked.
	} 
	If at stage $n$ the term exceeds the tail
	\begin{equation*}
		\xs{n}>\Xs{n}
	\end{equation*}
	and \clint{a}{b} is a component of \Cuts{n-1}, then \clint{a}{a+\Xs{n}} and \clint{b-\Xs{n}}{b}
	are disjoint components of \Cuts{n}.
\end{lemma}

\begin{prop}\label{prop:gaps}
	Suppose \seqx{} is a positive, non-increasing summable sequence.

	Then \Sigsum{}  has
	\begin{enumerate}
		\item\label{prop:gaps1} infinitely many components if the term exceeds the tail infinitely often;
		\item\label{prop:gaps2} at least $2^{N}$ components if the term exceeds the tail $N$ times.
	\end{enumerate}
\end{prop}
Note that \eqref{prop:gaps2} generalizes the lower bound given in \refer{prop}{eventtailbds}(\ref{prop:eventtailbds3}).

In particular,
\begin{corollary}\label{cor:gaps}
	The subsum set of a positive, non-increasing
	summable sequence \seqx{} is a finite union of intervals
	if and only if the tail eventually exceeds the term.
\end{corollary}

\refer{prop}{gaps}\eqref{prop:gaps1} and \refer{cor}{gaps} strongly suggest that a subsum set is either a finite union of closed
intervals or a Cantor set.  However, to show that a subsum set is a Cantor set, we need to show not only that
it has infinitely many components, but also that it has empty interior, or equivalently, that every component
is a single point.  The following observation, which follows from \refer{lem}{gaps}, suggests that this might
be true:
\begin{remark}\label{rmk:gaps}
	Suppose \seqx{} is a positive, non-increasing summable sequence.
	If the term exceeds the tail infinitely often, then each endpoint of every component of each \Cuts{n}
	constitutes a one-point component of $\Sigsum=\Cutinf$.
\end{remark}
Kakeya \cite{Kakeya} conjectured that the subsum set is a Cantor set if and only if the term exceeds the tail
infinitely often. Initially, I had the same intuition.
However, it turns out that there exist subsum sets with infinitely many components but nonempty interior.
We shall study some examples in the next section.

\section{Cantorvals}\label{sec:Cantorvals}
The following example was analyzed by Guthrie and Nymann
in \cite{Guthrie-Nymann} in the process of characterizing the range of an arbitrary finite measure.  
Consider the positive decreasing summable sequence 
$
	\frac{3}{4}, \frac{2}{4}, \frac{3}{16}, \frac{2}{16},\dots
$
that is,
\begin{align*}
	\xs{2k-1}&=\frac{3}{4^{k}}\\
	\xs{2k}&=\frac{2}{4^{k}}.
\end{align*}
The tails of this sequence are
\begin{align*}\label{eqn:tailsex}
	\Xs{2k}&=\frac{5}{3\cdot4^{k}}, k=0,1,\dots\\
	\Xs{2k-1}&=\frac{11}{3\cdot4^{k}}, k=1,\dots.
\end{align*}
Since
$3<\frac{11}{3}$ and $2>\frac{5}{3}$,
we see that every even-numbered term exceeds the corresponding tail, so \Sigsum{} has infinitely many components, by \refer{prop}{gaps}.

Guthrie and Nymann show that the subsum set contains the interval \clint{\frac{3}{4}}{1},
 but their argument (and example) can be seen as a special case of a number-theoretic argument shown me by
 Rick Kenyon, in the context of an example he sent me before I ran across \cite{Guthrie-Nymann}, namely
%\begin{equation*}
$
	\frac{6}{4},\frac{1}{4},\frac{6}{16},\frac{1}{16},\dots
$
%\end{equation*}
or
\begin{align*}
	\xs{2k-1}&=\frac{6}{4^{k}}\\
	\xs{2k}&=\frac{1}{4^{k}}.
\end{align*}
We note that this order, while it makes transparent the generating formulas for the sequence, is not 
monotone:  for  example, $\xs{2}=\frac{1}{4}=\frac{4}{16}<\frac{6}{16}=\xs{3}$, 
and $\xs{4}=\recip{16}=\frac{2}{2}>\xs{5}=\frac{6}{64}=\frac{3}{32}$.  For the record, the non-increasing
order is
%\begin{equation*}
$
	\frac{6}{4},\frac{6}{16},\frac{1}{4},\frac{6}{64},\frac{1}{16},\frac{6}{256},\frac{1}{64},\dots
$.
%\end{equation*}
The reader can verify that the term exceeds the tail infinitely often.  
%However, an interesting 
%number-theoretic
The following argument, suggested by Kenyon \cite[\S2]{Kenyon}, 
gives a way to generate many examples with
nonempty interior and, presumably, infinitely many components (including the Guthrie-Nymann one).%
\footnote{\cite[p. 515]{RJones} gives another example, which he attributes to Dan Velleman, very much
in the same spirit.}

The key observation (in the case of Kenyon's example above) is that every congruence class mod $4$ can be 
obtained as a sum of the ``digits'' $6$ and $1$, since $6\equiv2\mod{4}$ and $6+1=7\equiv3\mod{4}$.
Thus the set of sums of Kenyon's sequence is the set of all reals which can be expressed as ``generalized
base 4 expansions'' using the ``digits'' $0,1,6$ and $7$:
\begin{equation*}
	\Sigsum=\setbld{\sum_{i=1}^{\infty}\frac{\as{i}}{4^{i}}}{\as{i}\in\single{0,1,6,7}}.
\end{equation*}

\begin{prop}[R. Kenyon] 
	Suppose we are given \inNat{n} and $n$ integers $\ds{0},\ds{1},\dots,\ds{n-1}$ such that 
	\begin{equation*}
		\ds{i}\equiv j\mod{n}.
	\end{equation*}
	Then the set of ``generalized base $n$ expansions'' using these ``digits'' 
	\begin{equation*}
	\setS=\setbld{\sum_{i=1}^{\infty}\frac{\as{i}}{n^{i}}}{\as{i}\in\single{\ds{0},\dots,\ds{n-1}}}
	\end{equation*}
	has nonempty interior. 
\end{prop}
\begin{proof}
	The first step is to confirm the somewhat optimistic intuition that, since the digits include representatives
	of all the congruence classes $\mod{n}$, the finite sums of the form
	\begin{equation*}
		\sum_{i=1}^{k}\frac{\as{i}}{n^{i}},\quad\as{i}\in\single{\ds{0},\dots,\ds{n-1}}
	\end{equation*}
	should, by analogy with the standard case $\ds{j}=j$, 
	have fractional parts that include all rational numbers of the form $\frac{a}{n^{k}}$.
	The ``obvious'' reasoning we might expect does not apply:  for example, 
	$\frac{1}{4}+\frac{2}{4^{2}}=\frac{6}{16}$ while $\frac{1}{4}+\frac{6}{4^{2}}=\frac{10}{16}$;
	the difference is not an integer 
	even though $6=2\mod{4}$.
	However, it \emph{is} true that \emph{different} expressions of this form have \emph{different} 
	fractional parts.  To see this, suppose we have two such sums with the same fractional part:
	\begin{equation*}
		\frac{\as{1}}{n}+\frac{\as{2}}{n^{2}}+\dots+\frac{\as{k}}{n^{k}}
			=\frac{\bs{1}}{n}+\frac{\bs{2}}{n^{2}}+\dots+\frac{\bs{k}}{n^{k}}+N
	\end{equation*}
	(where each \as{i} and \bs{i} is one of our ``digits'' $\ds{0},\dots,\ds{n-1}$, and \inNat{N}).
	We can rewrite this as
	\begin{equation*}
		\frac{\as{1}-\bs{1}}{n^{1}}+\frac{\as{2}-\bs{2}}{n^{2}}+\dots\frac{\as{k}-\bs{k}}{n^{k}}=N
	\end{equation*}
	and multiply both sides by $n^{k}$:
	\begin{equation*}
		n^{k-1}(\as{1}-\bs{1})+n^{k-2}(\as{2}-\bs{2})+\dots+n(\as{k-1}-\bs{k-1})+(\as{k}-\bs{k})=n^{k}N.
	\end{equation*}
	Taking the congruence class of both sides $\mod{n}$, we get
	\begin{equation*}
		\as{k}-\bs{k}\equiv0\mod{n}.
	\end{equation*}
	But since the possible digits belong to different congruence classes $\mod{n}$, we must have
	\begin{equation*}
		\as{k}=\bs{k}.
	\end{equation*} 
	Thus by induction on $k$, $\as{i}=\bs{i}$ for $i=1.2,\dots,k$.
	
	Now, for a given (fixed) $k$, there are $n^{k}$ sums of the form 
	\begin{equation*}
		\sum_{i=1}^{k}\frac{\as{i}}{n^{i}}
	\end{equation*}
	as well as $n^{k}$ fractions of the form $\frac{a}{n^{k}}$ with $0\leq a<n^{k}$.
	Hence by the pigeonhole principle, congruence $\mod{n}$ generates a bijection
	between the two sets, confirming our intuition.
	
	The second step is then to reinterpret this statement to say that the integer translates of \setS{}
	cover the whole real line
	\begin{equation*}
		\bigcup_{k\in\Integers}\left(k+\setS\right)=\Reals.
	\end{equation*} 
	
	Finally, we invoke the \emph{Baire Category Theorem}, which in our context says that if a countable 
	union of sets equals \Reals{} then at least one of them has nonempty interior \cite{Baire}.%
	\footnote{This was Baire's doctoral dissertation;  
	see Dunham's highly readable account in \cite[pp. 184-191]{Dunham}.
	A more general version of this (involving complete metric spaces), is proved in many basic analysis 
	texts;  for example, 
	see  \cite[Thm. 4.31, pp. 243-5]{Pugh} or \cite[Prob. 16, p.40]{Rudin}, 
	} 
	From this we conclude that for at least one integer $k$, $\left(k+\setS\right)$ 
	has non-empty interior---but since it is a translate of \setS{}, the same is true of \setS{}.
\end{proof} 

Having established the existence of subsum sets with infinitely many components but non-empty interior,
we should try to understand better the structure of these sets.  

Suppose a subsum set \Sigsum{} has infinitely many components but non-empty interior.  
For each $n$, we can write \Sigsum{} as the union of $2^{n}$ translates of the set \Sigsumsn{}.
Invoking the Baire Category Theorem again (this time in its weaker form, involving a finite union)
we conclude that one, and hence all, of these translates has non-empty interior.  In particular, 
each interval \Js{\xis{1}\cdots\xis{n}} contains a subinterval of \Sigsum{}.  This means that every point
of \Sigsum{} is within distance \Xs{n} of some subinterval of \Sigsum{}.  Since $\Xs{n}\to0$, the 
subintervals (in particular the non-trivial components) of \Sigsum{} are dense.  At the same time, 
\refer{rmk}{gaps} tells us that the trivial (\ie{} one-point) components of \Sigsum{} are also dense,
in the sense that every endpoint of a non-trivial component is an accumulation point of trivial components.
In addition to Guthrie and Nymann \cite{Guthrie-Nymann}, such sets were studied by Mendes and Oliveira
\cite{Mendes-Oliveira}, in connection with the structure of arithmetic sums of Cantor sets 
(motivated by the study of bifurcation phenomena in dynamical systems).  
They dubbed them \emph{Cantorvals}.  In their context, three varieties of Cantorvals can arise, but 
because of the symmetry of subsum sets, the only kind that arises in our context is what they call an
\emph{$M$-Cantorval}.  I prefer the more descriptive term \emph{symmetric Cantorval}.
Formally:
\begin{definition}\label{dfn:Cantorval}
	A \deffont{symmetric Cantorval} is a nonempty compact subset \setS{} of the real line such that
	\begin{enumerate}
		\item \setS{} is the closure of its interior (\ie{} the nontrivial components are dense)
		\item Both endpoints of any nontrivial component of \setS{} are accumulation points of trivial 
		(\ie{} one-point) components of \setS.
	\end{enumerate}
\end{definition}

The remarks above establish a full topological classification of subsum sets for summable positive sequences, 
proven by Guthrie and Nymann
(with different terminology) in \cite{Guthrie-Nymann}:
\begin{theorem}[Guthrie-Nymann]\label{thm:classify}
	The subsum set of a positive summable sequence is one of the following:
	\begin{enumerate}
		\item a finite union of (disjoint) closed intervals;
		\item a Cantor set;
		\item a symmetric Cantorval.
	\end{enumerate}
\end{theorem}

Each of the first two categories in \refer{thm}{classify} provides a list of possible topological types: in the first
case, the number of components determines the topological type, while in the second, all Cantor sets are
homeomorphic, by a well-known theorem (see for example \cite[pp. 103-4]{Pugh}).  It turns out that all
(symmetric) Cantorvals are also homeomorphic.  This was proved in \cite{Guthrie-Nymann} and stated 
without explicit proof in \cite{Mendes-Oliveira}.

\begin{prop}\label{prop:Cantorvalhomeo}
	Any two symmetric Cantorvals are homeomorphic.
\end{prop}
\begin{proof}
	Given two Cantorvals \Cantor{} and \Cantorp, first identify the longest component of each;  if there is
	some ambiguity (because several components have the same maximal length), then pick the leftmost
	one.  There is a unique affine, order-preserving homeomorphism between them.
	
	Note that by definition there are other components of \Cantor{} \resp{\Cantorp} on either side of the
	chosen one.  In particular, its complement is contained in two disjoint intervals, one to the right
	and one to the left, and the part of each Cantorval in each of these intervals is again a Cantorval.
	Thus, we can apply the same algorithm to pair the longest nontrivial component to the left \resp{right}
	of the chosen one in \Cantor{} with the corresponding one in \Cantorp. Continuing in this way, we
	get an order-preserving correspondence between the non-trivial components of \Cantor{}
	and those of \Cantorp{}, and an order-preserving homeomorphism between corresponding components.
	But this means we have an order-preserving continuous mapping from the (dense) interior of 
	\Cantor{} onto the interior of \Cantorp.  This uniquely extends to a homeomorphism from all of \Cantor{}
	onto all of \Cantorp.
\end{proof}

Guthrie and Nymann point out that a model symmetric Cantorval can be constructed by following the standard
construction of the middle-third Cantor set (removing the middle third of each component at a given stage)
but then going back and ``filling in'' the gaps at \emph{every other} stage.

 \section{Bi-Geometric Sequences}\label{sec:bigeom} 
 
We saw in \refer{sec}{termstails} that the subsum set of a geometric sequence is either an interval or a 
Cantor set, because the relation between the term and the tail is always the same.  We can construct
examples which exhibit any particular pattern of alternation between the two possible relations by looking at the sequence of sets \Cuts{n} in a 
different way, in terms of ratios.

To be precise, given a sequence \seqx{} of terms, let us look at the associated sequence of tails,
\seqX, and for each index $i$, consider the proportion of \Xs{i} taken up by
\xs{i+1}:
\begin{equation*}
	\rhos{i}\eqdef\frac{\xs{i+1}}{\Xs{i}}
\end{equation*}
	or equivalently
\begin{equation}\label{eqn:xratio}
	\xs{i+1}=\rhos{i}\cdot\Xs{i}
\end{equation}
Then, since 
\begin{equation*}
	\Xs{i}=\xs{i+1}+\Xs{i+1},
\end{equation*}
we have
\begin{equation}\label{eqn:Xratio}
	\Xs{i+1}=(1-\rhos{i})\cdot\Xs{i}.
\end{equation}
Conversely, the sequence of ratios \single{\rhos{i}} together with the total sum \Xs{0} determines
the sequence \seqx{} recursively, via the initial condition 
\begin{align*}
	\xs{1}&=\rhos{0}\Xs{0}\\
	\intertext{and the relation}
	\xs{i+1}&=\rhos{i}\left(\rhos{i-1}^{-1}-1\right)\xs{i}.
\end{align*}
Equivalently, \xs{i} can be given by an explicit formula:
\begin{equation}\label{eqn:xratios}
	\xs{i}=\rhos{i}\prod_{j=0,\dots,i-1}(1-\rhos{j})\Xs{0}.
\end{equation}
The initial (total) sum \Xs{0} is simply a scaling factor, so to determine what kind of set occurs we
can \emph{assume that the total sum is }$\Xs{0}=1$.	

Now, at each stage, the term and tail are determined from the previous tail by \eqref{eqn:xratio}
and \eqref{eqn:Xratio}, from which it is easy to see that
\begin{itemize}
	\item the sequence \seqx{} is non-increasing if and only if for every $i$
	\begin{equation}\label{eqn:nondecr}
		\rhos{i}\leq\frac{\rhos{i-1}}{1-\rhos{i-1}};
	\end{equation}
	\item the \nth{} term exceeds the \nth{} tail ($\xs{n}>\Xs{n}$) if and only if%
	\footnote{In view of \refer{thm}{event1}, by picking an increasing sequence of ratios converging to \half{}
	at an appropriate rate, we can create sequences whose subsum set is a Cantor set of any desired 
	Lebesgue measure
	$m<1$.
	}
	\begin{equation*}
		\rhos{n-1}<\recip{2}
	\end{equation*}
	and (equivalently)
	\item the \nth{} tail bounds the \nth{} term ($\xs{n}\leq\Xs{n}$) if and only if
	\begin{equation*}
		\rhos{n-1}\geq\half.
	\end{equation*}
\end{itemize}

So one way to create a sequence for which both possibilities occur infinitely often is to pick two
ratios, 
\begin{equation}\label{eqn:twoparam}
	0<\alpha<\half<\beta<1,
\end{equation} 
and to set
\begin{equation*}
	\rhos{i}=\begin{cases}
			 \alpha     & \text{ for even }i, \\
			 \beta     & \text{ for odd }i.
			\end{cases}
\end{equation*}
This leads to the sequence
\begin{equation}\label{eqn:bigeom}
	\begin{split}
		\xs{2k}&=\beta(1-\alpha)^{k}(1-\beta)^{k-1}\\
		\xs{2k+1}&=\alpha(1-\alpha)^{k}(1-\beta)^{k}.
	\end{split}
\end{equation}

A sequence defined in this way spiritually resembles a geometric sequence, except that it involves 
two distinct ratios, so we might refer to it as a \deffont{bi-geometric sequence}.  
\footnote{An obvious generalization of this idea, which could be called a \emph{multi-geometric sequence}, 
is one where the sequence of ratios \rhos{i} is periodic;  we could refer to a sequence for which 
$\rhos{i+m}=\rhos{i}$ 
for some fixed $m>0$ and all $i$ as an \emph{$m$-geometric sequence}.  
We shall deal only with bi-geometric sequences in this paper.}

We have seen three examples of bi-geometric sequences earlier in this paper.
The sequence 
defined by \refer{eqn}{alternate}
was constructed so that
\begin{equation*}
	\alpha=\frac{2}{5},\quad\beta=\frac{3}{5}
\end{equation*}
while both the Guthrie-Nymann and Kenyon examples have
\begin{equation*}
	\alpha=\frac{9}{20},\quad\beta=\frac{6}{11}.
\end{equation*}

The first observation above says that, in order to have a non-increasing sequence \seqx{},
we also need \alp{} and \bet{} to satisfy
\begin{align}
	\alpha&\leq\frac{\beta}{1-\beta}\label{eqn:nonincratio1}\\
	\beta&\leq\frac{\alpha}{1-\alpha}\label{eqn:nonincratio2}.
\end{align}
Note that, since we require $\beta>\half$, \refer{eqn}{nonincratio2} puts further limitations 
on the possible values of \alp{}:
\begin{align*}
	\half&<\beta\leq\frac{\alpha}{1-\alpha}\\
	\intertext{forces}
	1-\alpha&<2\alpha
\end{align*}
or
\begin{equation}\label{eqn:alp}
	\alpha>\recip{3}.
\end{equation}
By contrast, \refer{eqn}{nonincratio1} puts no further restrictions on \bet{}.

We can try to analyze the subsum set of a bi-geometric sequence by using the idea of an 
\emph{iterated function system} (\cite{Barnsley}, \cite{Fractals}).  Suppose we have
a sequence defined in terms of two parameters $\recip{3}<\alpha<\half<\beta<1$,
subject to \eqref{eqn:nonincratio1} and \eqref{eqn:nonincratio2},  by \refer{eqn}{bigeom}.
The sets \Cuts{0} and \Cuts{1} are the same, and we can describe the set \Cuts{2} in terms of 
the set $\Cuts{0}=\clint{0}{1}$ as the union of four intervals \Js{ij}, $i,j\in\single{0,1}$,
each of length $\Xs{2}=(1-\alpha)(1-\beta)$, with respective endpoints 
\begin{align*}
	\xs{00}&=0\\
	\xs{10}&=\xs{1}=\alpha\\
	\xs{01}&=\xs{2}=\beta(1-\alpha)\\
	\xs{11}&=\xs{1}+\xs{2}=\alpha+\beta-\alpha\beta.\\
\end{align*}
Each of these intervals can be obtained from the basic interval $\Cuts{0}=\clint{0}{1}$
by scaling and translation;  specifically, we can define four affine functions, all with the same scaling factor 
\begin{equation*}
	\lambda=(1-\alpha)(1-\beta):
\end{equation*}
\begin{align*}
	\vphisof{00}{x}&=\lambda x\\
	\vphisof{01}{x}&=\xs{2}+\lambda x\\
		&=\beta(1-\alpha)+\lambda x\\
	\vphisof{10}{x}&=\xs{10} +\lambda x\\
		&=\alpha +\lambda x\\
	\vphisof{11}{x}&=\xs{1}+\xs{2}+\lambda x\\
		&\alpha+(1-\alpha)\beta+\lambda x.
\end{align*}
Then it is easy to see that, in terms of our earlier notation,
\begin{equation*}
	\Js{\xis{i}\xis{j}}=\vphisof{ij}{\Cuts{0}}.
\end{equation*}
But  our recursive relations for \xs{k} and \Xs{k} repeat every two steps, and hence we 
get recursive definitions of the sets \Cuts{k} and \Js{\xi}: 
\begin{align*}
	\Cuts{2k}=\Cuts{2k-1}&=\bigcup_{i,j=0}^{1}\vphisof{ij}{\Cuts{2k-1}=\Cuts{2k-2}};\\
	\intertext{more specifically, for each word \xie{} of length $2k$ in zeroes and ones,
	if  its initial $2k-2$-word is $\tilde\xi$ and last two entries are $i,j\in\single{0,1}$, then}
	\Js{\xi=\tilde\xi ij}&=\vphisof{ij}{\tilde\xi}.
\end{align*}

The various overlaps between images of \vphis{01} and \vphis{10} make it difficult to carry out a
careful analysis of the sets \Cuts{n} in general.  However, one easy observation allows us
to conclude in certain cases that the set \Cuts{\infty} is a Cantor set.  At each stage, the
set-mapping
\begin{equation*}
	\Cuts{2k}\mapsto\Cuts{2k+2}=
		\vphisof{00}{\Cuts{2k}}\cup\vphisof{01}{\Cuts{2k}}
			\cup\vphisof{10}{\Cuts{2k}}\cup\vphisof{11}{\Cuts{2k}}
\end{equation*}
first scales \Cuts{2k} by the factor $\lambda=(1-\alpha)(1-\beta)$, 
duplicates four copies of the scaled version,
then lays them down (with some overlap).  Ignoring the overlap,
we can assert that the total of the lengths of the intervals making up \Cuts{2k+2} 
is less than $4\lambda$ times the corresponding measure for \Cuts{2k}.  In particular, 
the longest interval in \Cuts{2k} will have length at most $(4\lambda)^{k}$.  
This allows us to formulate
\begin{remark}
	Suppose \seqx{} is a bi-geometric sequence with ratios
	\begin{equation*}
		0<\alpha<\half<\beta<1
	\end{equation*}
	satisfying  \refer{eqn}{nonincratio2}. 
	
	If
	\begin{equation}\label{eqn:lambda}
		\lambda\eqdef(1-\alpha)(1-\beta)<\recip{4},
	\end{equation}
	then \Sigsum{} is a Cantor set.
\end{remark}

This shows in particular that our first example
yields a Cantor set,
since $\lambda=\frac{6}{25}<\recip{4}$.  By contrast, the Guthrie-Nymann and Kenyon examples both have
$\lambda=\frac{11}{20}\cdot\frac{5}{11}=\recip{4}$.

In \refer{fig}{bigeom} we have sketched the parameter space $(\alpha,\beta)\in\clint{0}{1}\times\clint{0}{1}$ for bi-geometric sequences.  
Our discussion above concerned the upper-left quarter of this square, \rect{0}{\half}{\half}{1}, 
characterized by the inequalities \eqref{eqn:twoparam}, but by interchanging the roles of \alp{} 
and \bet{} where necessary we can extend it to the whole square.
The hatched areas are excluded by the requirement that the sequence \seqx{} be non-decreasing
(\refer{eqn}{nonincratio1} and \eqref{eqn:nonincratio2}).
The upper gray area is where \refer{eqn}{lambda} holds,
guaranteeing that \Sigsum{} is a Cantor set.  
Note that the two examples of Cantorvals (Guthrie-Nymann and Kenyon)
both correspond to a point on the boundary of this region, where $\lambda=\recip{4}$.

The lower gray area is where both \alp{} and \bet{} are at most equal to \half{}, which means the
tail always bounds the term---so $\Cuts{n}=\clint{0}{1}$ for all $n$.
This leaves the two white
regions (labeled with a question marks) where one ratio is at most \half{} 
while the other is greater than \half{}, 
where our analysis so far cannot completely determine the topology of the subsum set;  however, we do
know that in this region the subsum set has infinitely many components, so for each bi-geometric sequence
coming from parameters in this interval, the subsum set is either a Cantor set
or a symmetric Cantorval.  However we have not developed a test to distinguish, in general, which 
possibility a particular example exhibits.  In fact, I don't know if there are Cantorval examples 
with $\lambda>\recip{4}$, or, in the other direction, if there are any bi-geometric sequences with parameters
in the white region which yield Cantor sets.

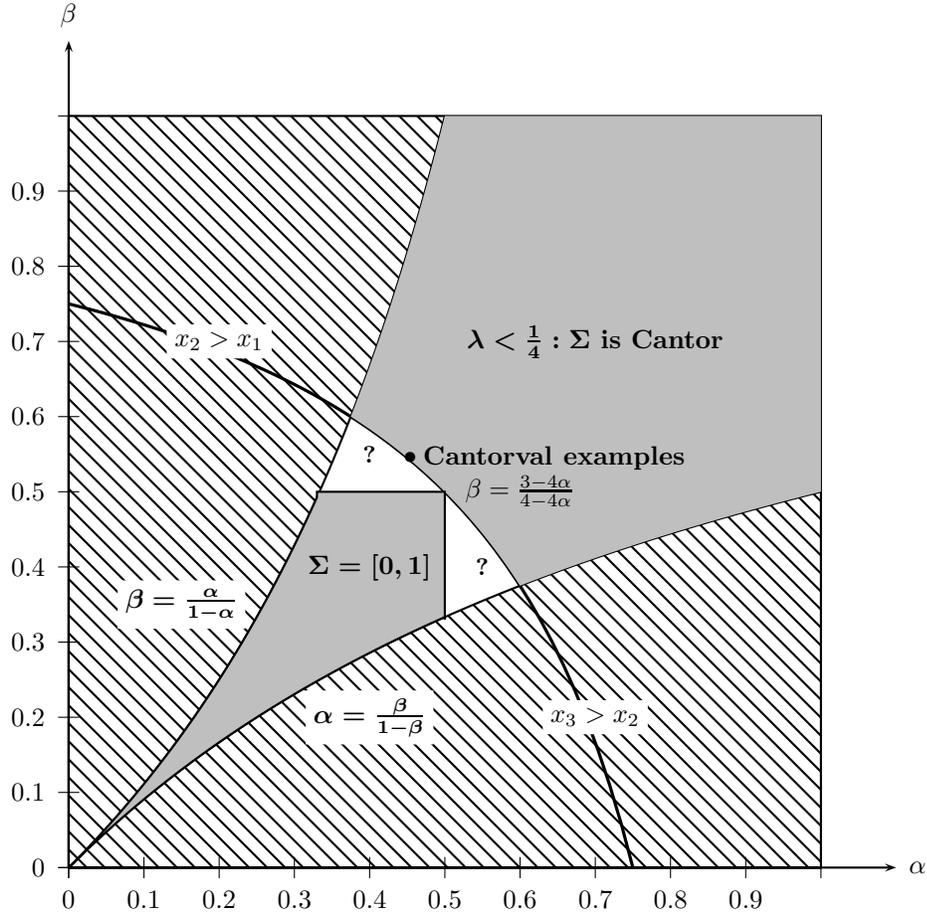
\begin{figure}[htbp]
\begin{center}
	\begin{pspicture}(-0.5,-0.5)(10,11.5)
	\psset{unit=10}
		\psaxes[axesstyle=frame, Dx=0.1, Dy=0.1](0,0)(1,1)
		\psline{->}(0,0)(1.1,0)\uput[r](1.1,0){$\alpha$}
		\psline{->}(0,0)(0,1.1)\uput[u](0,1.1){$\beta$}

		\pscustom[linestyle=none, fillstyle=vlines, fillcolor=gray]{
			\psplot{0}{0.5}{x 1 x sub div}
			\psline(0.5,1.0)(0,1.0)(0,0.5)
		}
		
		\pscustom[linestyle=none, fillstyle=vlines, fillcolor=gray]{
			\psplot{1.0}{0}{x 1 x add div}
			\psline(0,0)(1,0)(1,0.5)
		}
		
		\psplot{0}{0.5}{x 1 x sub div}
		\psplot[linewidth=1.2pt]{0}{0.75}{3 4 x mul sub 4 4 x mul sub div}
		\psplot{0}{1}{x 1 x add div}
		\rput*(0.4,0.2){$\boldsymbol{\alpha=\frac{\beta}{1-\beta}}$}
		
		\rput*(0.15,0.35){$\boldsymbol{\beta=\frac{\alpha}{1-\alpha}}$}
		
		\pscustom[fillstyle=solid, fillcolor=lightgray, linestyle=none]{%
			\psplot{0.5}{.375}{x 1 x sub div}
			\psplot{.375}{0.6}{3 4 x mul sub 4 4 x mul sub div}
			\psplot{0.6}{1.0}{x 1 x add div}
			\psline(1,0.5)(1,1)(0.5,1)
		}
		
		\pscustom[fillstyle=solid, fillcolor=lightgray]{
			\psplot{0}{0.5}{x 1 x add div}
			\psline(0.5,0.33)(0.5,0.5)(0.33,0.5)
			\psplot{0.33}{0}{x 1 x sub div}
		}

		\rput(0.4,0.55){\textbf{\small{?}}}
		\rput(0.55,0.4){\textbf{\small{?}}}
		\rput(0.4,0.4){$\boldsymbol{\Sigma=\clint{0}{1}}$}
		\rput(0.7,0.7){$\boldsymbol{\lambda<\recip{4}: \Sigma}$\textbf{ is Cantor}}
		
		\rput(0.454,0.545){$\bullet$}
		\uput[r](0.454,0.545){\textbf{Cantorval examples}}

		\rput*(0.2,0.7){$\xs{2}>\xs{1}$}
		\rput*(0.7,0.2){$\xs{3}>\xs{2}$}
		\rput(0.6,0.5){$\beta=\frac{3-4\alpha}{4-4\alpha}$}
		
	\end{pspicture}
\caption{Bi-geometric sequences}
\label{fig:bigeom}
\end{center}
\end{figure}
\section{Sequences with Varying Sign}\label{sec:vary}

We turn now to the general case, when some terms are positive while others are negative.  
Here we take advantage of another observation, given by
Riemann in \cite{Riemann:trig}  but attributed by him to Dirichlet: let us separate out the positive terms of \xs{i} as
\single{\xsp{i}} and the negative terms as \single{\xsm{i}}.  
Since the terms of each of these two sums have constant sign, we can define
\begin{align*}
	\sum\xsp{i}&=X^{+}\in\clint{0}{\infty}\\
	\sum\xsm{i}&=X^{-}\in\clint{-\infty}{0}.
\end{align*}

We can distinguish three possible configurations:

\begin{itemize}
	\item If both $X^{-}$ and $X^{+}$ are finite, the sequence is \textbf{absolutely summable}
	($\sum_{i=0}^{\infty}\abs{\xs{i}}	$ converges), because 
	\begin{equation*}
		\sum_{i=0}^{\infty}\abs{\xs{i}}\leq\abs{X^{-}}+X^{+}.
	\end{equation*}
	Recall that as a consequence every reordering of the sequence sums to the same
	(finite) number.
	
	\item If both $X^{-}$ and $X^{+}$ are infinite, the sequence is \textbf{conditionally summable}.
	It is a standard fact (attributed to Riemann) that 
	if a series converges while the corresponding series of absolute
	values diverges, then by rearranging the order of the terms we can get a series summing to
	any real number, or diverging to either $+\infty$ or $-\infty$. 
	Riemann's informal proof of this fact \cite[\S3]{Riemann:trig} rests on the observation
	that in this case both $X^{-}$ and $X^{+}$ are infinite.
	
	\item If one is finite and the other infinite, we will call the sequence \textbf{unconditionally 
	unsummable}.  
	In this case, \emph{every} reordering gives rise to a divergent series;  for example, if
	$X^{+}=\infty$ and $X^{-}$ is finite, then a partial sum of positive terms can be made
	arbitrarily large, while including  negative terms as well can at worst lower this sum
	by \abs{X^{-}}, so any rearrangement diverges to $\infty$.
	
\end{itemize}

In the absolutely summable case, Kakeya \cite{Kakeya} stated without proof that \Sigsum{} equals the 
interval \clint{X^{-}}{X^{+}} if and only if all the tails bound the sums for the sequence of absolute values
\single{\abs{\xs{k}}}.  Hornich \cite{Hornich} took this further: again assuming that the sequence is absolutely 
summable (so both $X^{\pm}$ are finite), and given a subsequence \single{\ys{i}} of our sequence, consider 
the translated sum of its absolute values
\begin{equation*}
	\sum\abs{\ys{i}}+X^{-}=\sum\abs{\ysp{i}}+\sum\abs{\ysm{i}}+X^{-}
		=\sum\ysp{i}+\sum\abs{\ysm{i}}-\sum_{k}\abs{\xsm{k}}
\end{equation*}
where the last summand is the sum of the absolute values of \emph{all} 
the \emph{negative} terms of the original sequence.
If we combine the last two sums, the terms \ys{i} in the subsequence get cancelled, leaving
the sum of all the negative terms which are \emph{excluded}
from the subsequence.  This of course is another subsum of our sequence.  Furthermore, \emph{every} subsum 
of the full sequence can be expressed in this way, which shows that the subsum set of the 
(absolutely summable) sequence \seqx{} is the translate by $X^{-}$ of the subsum set of the sequence
\single{\abs{\xs{i}}} of absolute values.  
\begin{prop}[Hornich]\label{prop:translate}
	If \seqx{} is an absolutely summable sequence, then
	\begin{equation*}
		\Sigsum=\Sigof{\single{\abs{\xs{k}}}_{k=1}^{\infty}}+X^{-}.
	\end{equation*}
\end{prop}

This means that the criteria we gave in \refer{thm}{event}, \refer{prop}{eventtailbds} and \refer{cor}{gaps}
can be applied to the (positive) sequence of absolute values to determine the topology of the subsum set of the 
original, variable sign but absolutely summable sequence.

Finally, if our original sequence is not absolutely summable, we can easily specify the subsum set.  In this
case we know that at least one of $X^{\pm}$ is infinite.  We concentrate on the case $X^{+}$ infinite;  the other
case is analogous.  Since the subsequence of positive terms is not summable, by \refer{thm}{posdiv}
$\Sigof{\single{\xsp{i}}}=\ropint{0}{\infty}$: we can obtain any positive number as the sum of a subsequence
of \emph{positive} terms.  If $X^{-}$ is finite, we can obtain any number in \ropint{X^{-}}{\infty} by 
adding a positive number to $X^{-}$;  if it is infinite, we can obtain any \emph{negative} number as the 
sum of some subsequence of \emph{negative} terms---so $\Sigsum=\Reals$ in this case.

With a little abuse of notation and sneaky reinterpretation, we can formulate a general characterization 
of all subsum sets. 

The abuse of notation is that we will allow closed interval notation with one or both endpoints infinite;  
it will be understood that in such a case the square bracket at that end 
should be replaced by a round parenthesis.

The sneaky reinterpretation is simply this:  if a positive sequence is not summable, then \emph{every} 
``tail'' is infinite, so bounds any term.

With these tweaks, we can state a general result, extending \refer{thm}{classify}:

\begin{theorem}
	Given a null sequence $\xs{k}\to0$, let $X^{+}$ \resp{$X^{-}$} be the (possibly infinite) sum of all the
	positive \resp{negative} terms.  Then the subsum set \Sigsum{} is a closed, perfect set whose 
	convex hull is the interval \clint{X^{-}}{X^{+}}, and which is symmetric with respect to reflection
	across the midpoint of this interval.
	
	Furthermore, denote the sequence of absolute values of our terms by 
	\begin{equation*}
		\as{k}=\abs{\xs{k}},\quad k=1,2,\dots
	\end{equation*}
	and its tails by
	\begin{equation*}
		\As{k}=\sum_{i>k}\as{k}.
	\end{equation*}
	Then:
	\begin{enumerate}
		\item If the tail bounds the term
		\begin{equation*}
			\as{k}\leq\As{k}
		\end{equation*}
		for all $k>K$, and the number of terms which exceed the tail
		\begin{equation*}
			\as{k}>\As{k}
		\end{equation*}
		is $N$, 
		then \Sigsum{} is the union of between $2^{N}$ and $2^{K}$ disjoint closed intervals.
		
		\item If the term exceeds the tail infinitely often, then \Sigsum{} is either a Cantor set
		or a symmetric Cantorval.  In particular, if the term always exceeds the tail, then \Sigsum{}
		is a Cantor set.
	\end{enumerate}
\end{theorem}

\section{Generalizations}
We comment briefly on two extensions of the material discussed in this paper.

First, Rafe Jones \cite{RJones} considers non-null real sequences.  
Several new phenomena are possible in this context.  If the sequence converges to a nonzero limit, then its
subsum set is a countable, unbounded set; in fact, \cite[Prop. 4.1]{RJones} any sequence 
possessing no null subsequences
has a countable subsum set.  
Jones notes \cite[p. 514]{RJones} that in general the subsum set of a non-null sequence need not be closed
(for example, \Sigof{\frac{n+1}{n}} has $1$ as an accumulation point, but does not contain it).
In general, the subsum set of any sequence is either meager 
(\ie{} of first Baire category, and hence totally disconnected), 
or else its interior is a dense subset \cite[Theorem 3.1]{RJones}. 
If it is neither countable nor an unbounded interval,
then it consists of a countable union of translates of some null subsequence \cite[Prop. 3.2]{RJones}.
  
A second extension, referenced by Jones, is the work of Manuel Mor\'an \cite{Moran1, Moran2} which considers
subsum sets of sequences in higher dimensions, in particular of complex sequences, under an assumption 
(``quick convergence'') analogous to our ``terms exceed tails'' condition.  
In this context, Mor\'an studies the Hausdorff dimension of the fractal sets generated by families sequences obtained
from analytic functions.

\bibliography{Subsequence}

\begin{thebibliography}{10}

\bibitem{Baire}
Ren\'e Baire.
\newblock {\em Sur les fonctions des variables r\'eelles}.
\newblock Imprimerie Bernardoni de C. Rebeschini \& Cie, 1899.

\bibitem{Riemann:Papers}
Roger Baker, Charles Christenson, and Henry~Orde (translators).
\newblock {\em Bernhard Riemann Collected Papers}.
\newblock Kendrick Press, 2004.

\bibitem{Barnsley}
Michael Barnsley.
\newblock {\em Fractals Everywhere}.
\newblock Academic Press, 1988.
\newblock Second Edition, Morgan Kufmann 1993 (Hardback), 2000 (Paperback).

\bibitem{Birkhoff:Sourcebook}
Garrett Birkhoff, editor.
\newblock {\em A Source Book in Classical Analysis}.
\newblock Harvard University Press, 1973.

\bibitem{Dirichlet:Beweis}
P.~G.~L. Dirichlet.
\newblock Beweis des {S}atzes, dass jede unbegrenzte arithmetische
  {P}rogression, deren erstes {G}lied und {D}ifferenz ganze {Z}ahlen ohne
  gemeinschaftlichen {F}actor sind, unendlich viele {P}rimzahlen enth\"alt.
\newblock {\em Abhandlungen der K\"oniglich Preussischen Akademie der
  Wissenschaften}, 8:45--81, 1837.
\newblock reprinted in \cite[pp. 313-342]{Dirichlet:Werke}.

\bibitem{Dunham}
William Dunham.
\newblock {\em The {C}alculus {G}allery: Masterpieces from Newton to Lebesgue}.
\newblock Princeton Univ. Press, 2005.

\bibitem{Fractals}
Gerald~A. Edgar, editor.
\newblock {\em Classics on Fractals}.
\newblock Westview Press, 2004.

\bibitem{Guthrie-Nymann}
J.~A. Guthrie and J.~E. Nymann.
\newblock The topological structure of the set of subsums of an infinite
  series.
\newblock {\em Colloquium Mathematicum}, 55:323--327, 19889.
\newblock MR 0978930 (90b: 40010).

\bibitem{Hornich}
Hans Hornich.
\newblock {\"U}ber beliebige {T}eilsummen absolute konvergenter {R}eihen.
\newblock {\em Montashefte f\"{u}r Mathematik und Physik}, 49:316--320, 1941.

\bibitem{RJones}
Rafe Jones.
\newblock The achievement sets of series.
\newblock {\em American Mathematical Monthly}, 118(6 (June-July)):508--521,
  2011.

\bibitem{Kakeya}
S.~Kakeya.
\newblock On the partial sums of an infinite series.
\newblock {\em Tohoku Sci. Rep.}, pages 159--163, 1915.

\bibitem{Kenyon}
Richard Kenyon.
\newblock Projecting the one-dimensional {S}ierpinski gasket.
\newblock {\em Israel J. Math.}, 97:221--238, 1997.

\bibitem{Dirichlet:Werke}
L.~Kronecker and L.Fuchs, editors.
\newblock {\em G. Lejeune Dirichlet's {Werke}. Herausgegeben auf Veranlassung
  der K\"oniglich Preussischen Akademie der Wissenschaften, von L. Kronecker}.
\newblock G. Reimer, 1889-97.
\newblock A number of more recent reprints of this collection are available.

\bibitem{Mendes-Oliveira}
Pedro Mendes and Fernando Oliveira.
\newblock On the topological structure of the arithmetic sum of two {C}antor
  sets.
\newblock {\em Nonlinearity}, pages 329--343, 1994.

\bibitem{Menon}
P.~Kesava Menon.
\newblock On a class of perfect sets.
\newblock {\em Bulletin, Amer. Math. Soc.}, 54:706--711, 1948.

\bibitem{Moran1}
Manuel Mor\'an.
\newblock Fractal series.
\newblock {\em Mathematika}, 36:334--348, 1989.

\bibitem{Moran2}
Manuel Mor\'an.
\newblock Dimension functions for fractal sets associated to series.
\newblock {\em Proceedings, Amer. Math. Soc.}, 120:749--754, 1994.

\bibitem{Nymann-Saenz}
J.~E. Nymann and Ricardo Saenz.
\newblock The topological structure of the set of $p$-sums of a sequence.
\newblock {\em Publ. Math. Debrecen}, 50:305--316, 1997.
\newblock MR 1446474 (98d:11013).

\bibitem{Pugh}
Charles~Chapman Pugh.
\newblock {\em Real Mathematical Analysis}.
\newblock Undergraduate Texts in Mathematics. Springer-Verlag, 2002.

\bibitem{Riemann:trig}
Bernhard Riemann.
\newblock \"{U}ber die {D}arstellbarkeit einer {F}unction durch eine
  trigonometrische {R}eihe (on the representability of a function by means of a
  trigononmetric series).
\newblock In Heinrich Weber, editor, {\em {G}esammelte {M}athematische {W}erke
  und {W}issentschaftlicher {N}achlass}, pages 227--264. Dover, 1953.
\newblock An English translation of part of this appears in \cite[pp.
  16-23]{Birkhoff:Sourcebook}; a full translation is included in \cite[pp.
  219-256]{Riemann:Papers}.

\bibitem{Rudin}
Walter Rudin.
\newblock {\em Principles of Mathematical Analysis}.
\newblock McGraw-Hill, 2 edition, 1968.

\end{thebibliography}
\bibliographystyle{plain}

\end{document}